\documentclass[reqno,12pt]{amsart}

\usepackage{amsmath, amsthm, amsfonts, amssymb}
\input{mathrsfs.sty}

\usepackage{amsmath}
\usepackage{amsfonts}
\usepackage{amssymb}
\usepackage{eufrak}
\usepackage{amscd}
\usepackage{amsthm}
\usepackage{amstext}
\usepackage[all]{xy}

\newcommand{\supp}{\operatorname{supp}}

   \theoremstyle{plain}
   \newtheorem{thm}{Theorem}[section]
   \newtheorem{prop}[thm]{Proposition}
   \newtheorem{lem}[thm]{Lemma}
   \newtheorem{cor}[thm]{Corollary}
   \theoremstyle{definition}

   \theoremstyle{remark}

 \numberwithin{equation}{section}

\author{V. Manuilov}

\date{}

\address{Moscow Center for Fundamental and Applied Mathematics, Moscow State University,
Leninskie Gory 1, Moscow, 
119991, Russia}

\email{manuilov@mech.math.msu.su}

\thanks{The author acknowledges support by the Russian Science Foundation grant No. 25-11-00018.}

\title{On ideals in the semilattice of coarse equivalence classes of metrics}

\begin{document}

\begin{abstract}
For a Hausdorff topology on the set of ideals of the semilattice $M(X)$ of coarse equivalence classes of metrics on a set $X$, the space $I(M(X))$ of ideals is the closure of the set of principal ideals, thus allowing to view non-principal ideals as generalizations of coarse equivalence classes of metrics. 

Some ideals arise from coarse structures on $X$. We define a map $\Phi$ from $I(M(X))$ to the set $CS(X)$ of coarse structures on $X$, and a map $\Psi$ backwards, and show that $\Psi\circ\Phi$ is the identity map, thus allowing to identify coarse structures with some ideals of $M(X)$. We show that there are ideals that do not come from $CS(X)$.

For any ideal $F$ we define the generalized uniform Roe algebra as the direct limit $C^*$-algebra of the uniform Roe algebras for the equivalence classes of metrics in the ideal, and show that it coincides with the uniform Roe algebra of $\Phi(F)$.
\end{abstract}

\maketitle

\section{Introduction}

Two metrics, $d_1$ and $d_2$, on a set $X$ are coarsely equivalent ($d_1\sim d_2$) if there exists a monotone function $\varphi:[0,\infty)\to[0,\infty)$ with $\lim_{t\to\infty}\varphi(t)=\infty$ such that 
$$
d_1(x,y)\leq\varphi(d_2(x,y)) \quad\mbox{and}\quad d_2(x,y)\leq\varphi(d_1(x,y))\quad \mbox{for\ any\ }x,y\in X. 
$$

Some geometric problems deal not with individual metrics, but with the space of all metrics, and often one has to complete this space by the limit points, which are not metrics. In this paper we complete not the space of metrics, but that of their coarse equivalence classes. To this end we use a natural join semilattice structure on the set $M(X)$ of equivalence classes of metrics, and view ideals in this join semilattice as generalized equivalence classes of metrics, while the equivalence classes of usual metrics correspond to the principal ideals. This completion is the closure of $M(X)$ with respect to a natural Hausdorff topology.

Let $I(M(X))$ denote the set of all ideals in the semilattice $M(X)$ of coarse equivalence classes of metrics on $X$.
Some ideals arise from coarse structures on $X$. We define a map $\Phi$ from $I(M(X))$ to the set $CS(X)$ of coarse structures on $X$, and a map $\Psi$ backwards, and show that $\Psi\circ\Phi$ is the identity map, thus allowing to identify coarse structures with certain ideals of $M(X)$. We show that there are ideals that do not come from $CS(X)$.

For any ideal $F\in I(M(X))$ (generalized metric) one can define the generalized uniform Roe algebra $C^*_u(X,F)$ as the direct limit $C^*$-algebra of the uniform Roe algebras for all equivalence classes of metrics in the ideal. On this way we loose the one-to-one correspondence between the equivalence classes of metrics and their uniform Roe algebras. We show that different ideals, $F_1$ and $F_2$, have the same generalized uniform Roe algebra iff these ideals correspond to the same coarse structure, i.e. if $\Phi(F_1)=\Phi(F_2)$. 

\section{Equivalence classes of metrics as a lattice}

Recall that a partially ordered set $L$ is a join semilattice if for any $\alpha,\beta\in L$ there exists their least upper bound, denoted by $\alpha\lor \beta$. 

Recall that a metric $d$ on $X$ is uniformly discrete if there exists $r>0$ such that $d(x,y)\geq r$ for any $x\neq y\in X$.

Let $\mathcal M(X)$ denote the set of all uniformly discrete metrics on a set $X$, and let $M(X)$ be the set of coarse equivalence classes of uniformly discrete metrics metrics on $X$, i.e. $M(X)=\mathcal M(X)/\sim$. For a metric $a$, we write $[a]$ for its coarse equivalence class.

Define a partial order on $\mathcal M(X)$ by setting $a\preceq b$, $a,b\in\mathcal M(X)$, if $a(x,y)\geq b(x,y)$ for any $x,y\in X$. 

Given $a,b\in\mathcal M(X)$, set 
$$
a\lor b(x,y)=\inf\Bigl\{\sum_{i=1}^n\min(a(x_{i-1},x_i),b(x_{i-1},x_i))\Bigr\},
$$  
where the infimum is over all sequences of points $x=x_0,x_1,\ldots,x_n=y$.

The following statement is folklore.

\begin{lem}
$a\lor b\in\mathcal M(X)$, and is the least upper bound for $a$ and $b$.

\end{lem}
\begin{proof}
Think of $X$ as a weighted complete graph such that the length of an edge $[x_i,x_j]$ equals $\min(a(x_i,x_j),b(x_i,x_j))$. Then $a\lor b(x,y)$ is the distance between $x$ and $y$ in this graph, hence it satisfies the triangle inequality. As both $a$ and $b$ are uniformly discrete, so is $a\lor b$.

Clearly, $a\lor b(x,y)\leq \min(a(x,y),b(x,y))\leq a(x,y)$ for any $x,y$, hence $a\preceq a\lor b$. Similarly, $b\preceq a\lor b$.

Let $c\in \mathcal M(X)$ satisfy $a\preceq c$ and $b\preceq c$. We have 
$$
c(x,y)=\inf\Bigl\{\sum_{i=1}^n c(x_{i-1},x_i)\Bigr\}\leq \inf\Bigl\{\sum_{i=1}^n\min(a(x_{i-1},x_i),b(x_{i-1},x_i))\Bigr\}=a\lor b(x,y),
$$  
hence $a\lor b\preceq c$.
\end{proof}

Thus, $\mathcal M(X)$ is a join semilattice with respect to $\lor$.

\begin{lem}
Let $a,a',b,b'\in\mathcal M(X)$, $a\sim a'$, $b\sim b'$. Then $a\lor b\sim a'\lor b'$. 

\end{lem}
\begin{proof}
By the assumption, there are homeomorphisms $\phi_1,\phi_2$ of $[0,\infty)$ such that $a'(x,y)\leq\phi_1(a(x,y))$ and $b'(x,y)\leq\phi_2(b(x,y))$. First, note that we may leave only one homeomorphism, e.g. $\varphi$ such that $\varphi(t)=\max(\phi_1(t),\phi_2(t))$. Second, note that without loss of generality we may assume that $\varphi$ satisfies $\varphi(u+v)\geq\varphi(u)+\varphi(v)$ for any $u,v\geq 0$. Then
\begin{eqnarray*}
a'\lor b'(x,y)&=&\inf\Bigl\{\sum_{i=1}^n\min(a'(x_{i-1},x_i),b'(x_{i-1},x_i))\Bigr\}\\
&\leq&
\inf\Bigl\{\sum_{i=1}^n\varphi(\min(a(x_{i-1},x_i),b(x_{i-1},x_i)))\Bigr\}\\
&\leq& 
\inf\varphi\Bigl(\sum_{i=1}^n\min(a(x_{i-1},x_i),b(x_{i-1},x_i))\Bigr)\\
&=&\varphi(a\lor b(x,y)).
\end{eqnarray*}  
Interchanging $a$ and $a'$, $b$ and $b'$, we get coarse equivalence of $a\lor b$ and $a'\lor b'$.
\end{proof}

Thus, $\lor$ is well defined on the equivalence classes, $[a]\lor[b]=[a\lor b]$, making $M(X)$ a join semilattice. 

\begin{lem}
Let $a,b\in \mathcal M(X)$. The following are equivalent: 
\begin{enumerate}
\item
$[a]\lor[b]=[a]$; 
\item
there exists a homeomorphism $\varphi$ of $[0,\infty)$ such that $a(x,y)\leq \varphi(b(x,y))$ for any $x,y\in X$. 

\end{enumerate}
\end{lem}
\begin{proof}
First, suppose that (2) holds. Without loss of generality we may assume that $\varphi(t)\geq t$ for any $t\in[0,\infty)$ and that $\varphi(u+v)\geq \varphi(u)+\varphi(v)$ for any $u,v\geq 0$.
As 
$$
\min(a(x,y),b(x,y))\geq \min(a(x,y),\varphi^{-1}(a(x,y)))=\varphi^{-1}(a(x,y)),
$$
we have
\begin{eqnarray*}
a\lor b(x,y)&=&\inf\Bigl\{\sum_{i=1}^n \min(a(x_{i-1},x_i),b(x_{i-1},x_i))\Bigr\}\\
&\geq& \inf\Bigl\{\sum_{i=1}^n \varphi^{-1}(a(x,y))\Bigr\}\\
&\geq&\varphi^{-1}\Bigl(\sum_{i=1}^n a(x_{i-1},x_i)\Bigr)\\
&\geq&\varphi^{-1}(a(x,y))
\end{eqnarray*}
for any $x,y\in X$. On the other hand, $a\lor b(x,y)\leq a(x,y)$ holds for any $x,y\in X$ as well. Thus, $a\lor b\sim a$.

Second, suppose that (1) holds. Then there exists a homeomorphism $\varphi$ of $[0,\infty)$ such that $a(x,y)\leq \varphi(a\lor b(x,y))$ for any $x,y\in X$.
Then $b(x,y)\geq a\lor b(x,y)\geq\varphi^{-1}(a(x,y))$ for any $x,y\in X$.
\end{proof}
 
Thus, the structures of a join semilattice and of a partially ordered set on $M(X)$ agree.

Note that, unlike $\mathcal M(X)$, $M(X)$ has the greatest element.  
Let $d_1(x,y)=1$ for any $x\neq y\in X$, $\delta_1=[d_1]\in M(X)$. 

\begin{lem}
$\alpha\preceq\delta_1$ for any $\alpha\in M(X)$.

\end{lem}
\begin{proof}
Obvious.
\end{proof}

\begin{thm}
The smallest element in $M(X)$ exists iff $X$ is countable.
\end{thm}

The proof consists of the following Lemmas \ref{L1}, \ref{L2} and \ref{L3}.

\begin{lem}\label{L1}
Let $X$ be uncountable. Then there is no smallest element in $M(X)$.

\end{lem}
\begin{proof}
Suppose the contrary: there exists a metric $d\in\mathcal M(X)$ such that $[d]\preceq\alpha$ for any $\alpha\in M(X)$. Set $E_N=\{(x,y)\in X\times X:d(x,y)\leq N\}$. Then $X\times X=\cup_{N\in\mathbb N} E_N$ is uncountable, hence there exist $N\in\mathbb N$ such that $E_N$ is infinite. Then there exists a sequence $\{(x_n,y_n)\}_{n\in\mathbb N}$ consisting of different pairs, such that $d(x_n,y_n)\leq N$ for any $n\in\mathbb N$. Passing to a subsequence, we may assume that either all points $x_n$ and $y_n$ are different, or there is $y_0\in X$ such that $d(x_n,y_0)\leq N$ for any $n\in\mathbb N$. There exists a uniformly discrete metric $a$ on $X$ such that $\lim_{n\to\infty}a(x_n,y_n)=\infty$ in the first case, or $\lim_{n\to\infty}a(x_n,y_0)=\infty$. Then $[d]\preceq[a]$ fails.
\end{proof}

We say that a metric $a\in\mathcal M(X)$ is \emph{uniformly sparse} if $\lim_{R\to\infty}f_a(R)=\infty$, where $f_a(R)=\inf\{a(x,y):x\neq y\in X, x,y\notin B_R(x_0,a)\}$, $x_0\in X$. Clearly, this definition does not depend on a choice of $x_0$ and is coarsely equivalent. If $X$ is countable then uniformly sparse proper metrics exist. For example, enumerate the points of $X$: $X=\{x_n\}_{n\in\mathbb N}$, and set $d_0(x_n,x_m)=|2^n-2^m|$, $\delta_0=[d_0]$.

\begin{lem}\label{L2}
$\delta_0\preceq\alpha$ for any $\alpha\in M(X)$.

\end{lem}
\begin{proof}
Let $a\in\mathcal M(X)$, $[a]=\alpha$. 
To show that $\delta_0\preceq\alpha$, we have to construct a homeomorphism $\varphi$ of $[0,\infty)$ such that $a(x_n,x_m)\leq\varphi(d_0(x_n,x_m))=\varphi(|2^n-2^m|)$. Define $\varphi(2^{m-1})$ by $\varphi(2^{m-1})=\max\{a(x_i,x_m):i\leq m-1\}$. Then $a(x_i,x_m)\leq \varphi(2^m-2^{m-1})\leq\varphi(2^m-2^i)=\varphi(d_0(x_i,x_m))$ for any $i<m$.
\end{proof} 

\begin{lem}\label{L3}
Let $X$ be countable, and let $d'\in\mathcal M(X)$ be a proper uniformly sparse metric. Then $d'\sim d_0$.

\end{lem}
\begin{proof}
We have already shown that $[d_0]\preceq[d']$, so it remains to show that $[d']\preceq[d_0]$.

Set $\phi(t)=\inf\{d'(x,y):d_0(x,y)\geq t\}$. Clearly, $\phi$ is monotonely non-decreasing, and $d'(x,y)\geq\phi(d_0(x,y))$ for any $x,y\in X$. 


It remains to show that $\lim_{t\to\infty}\phi(t)=\infty$. Suppose the contrary: there exists $C>0$ and a sequence $(x_{n_k},y_{n_k})_{k\in\mathbb N}$ such that $d_0(x_{n_k},y_{n_k})>n$ and $d'(x_{n_k},y_{n_k})\leq C$. Properness of $d'$, together with $d'(x_{n_k},y_{n_k})\leq C$, implies that both sequences go to infinity with respect to the metric $d'$. But the latter contradicts uniform sparseness of $d'$.
\end{proof}

Besides the join $\lor$, there is a meet operation $\land$ on $M(X)$. For $a,b\in\mathcal M(X)$ it is defined either by $a\land b(x,y)=\max(a(x,y),b(x,y))$ or $a\land b(x,y)=a(x,y)+b(x,y)$. Passing to the coarse equivalence classes, either of these formulas defines the meet $\land$ by $[a]\land[b]=[a\land b]$. Clearly, $\alpha\land\beta\preceq\alpha,\beta$ for any $\alpha,\beta\in M(X)$. Thus, $M(X)$ is a lattice. Note that this lattice is not distributive.

\section{Ideals in semilattices}

Given a join semilattice $L$, a subset $F\subset L$ is \emph{downwards closed} if $a\preceq b$ and $b\in F$ implies that $a\in F$. $F\subset L$ is a subsemilattice if $a,b\in L$ implies $a\lor b\in F$. 

A subset $F\subset F$ is an \emph{ideal} if $F$ is both downwards closed and a subsemilattice, \cite{book}.
Ideals are duals for filters.

Let $I(L)$ denote the set of all ideals of $L$.

For $a\in L$, set $F_a=\{b\in L:b\preceq a\}$. By transitivity of $\preceq$, $F_a$ is a downwards closed subset. Also if $b,c\preceq a$ then $b\lor c\preceq a$. This gives an inclusion $L\to I(L)$, $a\mapsto F_a$. Ideals $F_a$, $a\in L$, are called principal ideals, \cite{book}. 

Note that if $F,G\in I(L)$ then $F\cap G\in I(L)$.

Characteristic functions provide a one-to-one correspondence between subsets of $L$ and maps from $L$ to $\{0,1\}$, so $I(L)$ can be considered as a subset of $\{0,1\}^L$. Endow the latter with the product topology. 

The following is folklore.

\begin{lem}
The subset $I(L)\subset\{0,1\}^L$ is closed.

\end{lem}
\begin{proof}
We have to show that the complement to $I(L)$ is an open set. For $G\in \{0,1\}^L$, the base of the product topology consists of the sets 
$$
U_{a_1,\ldots,a_n;b_1,\ldots,b_m}(G)=\{F\in \{0,1\}^L: a_1,\ldots,a_n\in F,b_1,\ldots,b_m\notin F\},
$$ 
where $a_i\in G$, $b_j\notin G$, $i=1,\ldots,n$, $j=1,\ldots,m$.

If $G\notin I(L)$ then there are two possibilities: 
\begin{enumerate}
\item
there exist $b\preceq a$ with $b\notin G$, $a\in G$;
\item
there exist $a,b\in G$ such that $a\lor b\notin G$.
\end{enumerate}
In the case (1), $U_{a;b}(G)\cap I(L)=\emptyset$, in the case (2), $U_{a,b;a\lor b}(G)\cap I(L)=\emptyset$.
\end{proof} 

\begin{cor}
$I(L)$ is a compact Hausdorff space.

\end{cor}

\begin{lem}
If $L$ has the smallest element $\delta_0$ then $\emptyset\in I(L)$ is an isolated point.

\end{lem}
\begin{proof}
Let $U_{;\delta_0}(\emptyset)=\{G\in I(L):\delta_0\notin G\}$ be a basic open set, and let $F\in I(L)$ is non-empty. Let $\alpha\in F$. Suppose that $F\in U_{;\delta_0}(\emptyset)$. As $\delta_0\preceq\alpha$ and $\alpha\in F$ then $\delta_0\in F$, but this contradicts $F\in U{;\delta_0}(\emptyset)$.  
\end{proof}

\begin{lem}
If $L$ is a lattice and has the smallest element $\delta_0$ then $I(L)\setminus\{\emptyset\}$ is a compactification for $L$.

\end{lem}
\begin{proof}
We have to show that $L$ is dense in $I(L)\setminus\{\emptyset\}$.

Let $F\in I(L)$, $F\neq\emptyset$, and let $U(F)$ be a basic open neighborhood of $F$. We have to show that there exists $\gamma\in L$ such that $F_\gamma\in U(F)$. First, consider $U(F)=U_{;\beta_1,\ldots,\beta_m}(F)$, where $b_j\notin F$, $j=1,\ldots,m$. Note that $\delta_0\in F$. Indeed, as $F$ is not empty, there exists $\alpha\in F$. Then $\delta_0\preceq\alpha$, hence $\delta_0\in F$. Then $F_{\delta_0}\in U_{;\beta_1,\ldots,\beta_m}(F)$.

Second, consider $U(F)=U_{\alpha_1,\ldots,\alpha_n;\beta_1,\ldots,\beta_m}(F)$, where $\alpha_i\in F$, $\beta_1,\ldots,\beta_m\notin F$ ($m$ can be zero here, but $n\geq 1$). Set $\gamma=\alpha_1\land\cdots\land\alpha_n$. As $\gamma\preceq\alpha_1$, $\gamma\in F$. Clearly, $\beta_j\notin F_\gamma$, $j=1,\ldots,m$. Otherwise we would have $\beta_j\preceq\gamma$, hence $\beta_j\preceq a_1$, hence $\beta_j\in F$. Therefore, $F_\gamma\in U_{\alpha_1,\ldots,\alpha_n;\beta_1,\ldots,\beta_m}(F)$.  \end{proof}

If $a\preceq b$ then $F_b\subset F_a$, and thus the partial order can be extended to $I(L)$ by setting $F\preceq F'$ if $F'\subset F$.

Note that if $\delta_0$ is the smallest element in $L$ then $F_{\delta_0}$ is the unique ideal in $L$ consisting of one element.

\section{Ideals from coarse structures}

Many ideals in $M(X)$ come from coarse structures. Our reference for coarse structures is \cite{Roe}. Recall that a coarse structure $\mathcal E$ on $X$ is a family of subsets of $X\times X$ such that
\begin{enumerate}
\item
the diagonal $\Delta_X=\{(x,x):x\in X\}$ lies in $\mathcal E$;
\item
$\mathcal E$ is closed under taking subsets, finite unions and inverses, where $E^{-1}=\{(y,x):(x,y)\in E\}$, $E\subset X\times X$;
\item
$\mathcal E$ is closed under finite compositions, where $E\circ G=\{(x,z):(x,y)\in E, (y,z)\in G \mbox{\ for\ some\ }y\in X\}$, $E,G\subset X\times X$.
\end{enumerate}

For a metric $d$ set $E_R(d)=\{(x,y)\in X\times X:d(x,y)\leq R\}$. Then the class of coarse equivalence of $d$ determines a coarse structure $\mathcal E_d$ on $X$: $E\in\mathcal E_d$ if there exists $R>0$ such that $E\subset E_R(d)$. A coarse structure is metrizable iff it is countably generated.

A metric $d$ on $X$ is dominated by a coarse structure $\mathcal E$ if $E_R(d)\in\mathcal E$ for any $R>0$. This obviously passes to the coarse equivalence classes of metrics.

Let $F^\mathcal E$ denote the set of equivalence classes of metrics on $X$ dominated by $\mathcal E$.

\begin{lem}
$F^\mathcal E$ is an ideal for any coarse structure $\mathcal E$.

\end{lem}
\begin{proof}
Let $a,b\in\mathcal M(X)$, $\alpha=[a]$, $\beta=[b]$. If $\alpha\preceq\beta$ and $b$ is dominated by $\mathcal E$ then $E_R(a)\subset E_R(b)$, hence $a$ is dominated by $\mathcal E$.

Let $a,b\in\mathcal M(X)$. As they are uniformly discrete, for any $R>0$ there exists $C_R>0$ such that for any $x,y$ with $a\lor b(x,y)\leq R$ there exists a sequence $x=x_0,x_1,\ldots,x_n=y$ with $n\leq C_R$ and $\min(a(x_{i-1},x_i),b(x_{i-1},x_i))<R+1$, $i=1,\ldots,n$.

Set $E=E_{R+1}(a)\cup E_{R+1}(b)$. Then $E\in\mathcal E$. If $(x_{i-1},x_i)\in E$ then $\min(a(x_{i-1},x_i),b(x_{i-1},x_i))\leq R+1$. If $a\lor b(x,y)\leq R$ then $(x,y)\in E\circ\cdots\circ E$ ($n$ times), hence $(x,y)\in\cup_{j=1}^{C_R} E^j\in \mathcal E$. Therefore, $E_R(a\lor b)\in\mathcal E$, so $a\lor b$ is dominated by $\mathcal E$.
\end{proof}

Denote the set of all coarse structures on $X$ by $CS(X)$.

Let $F\in I(M(X))$. Define $\Phi(F)\in CS(X)$ as the coarse structure generated by all $\mathcal E_d$, $[d]\in F$.
For $\mathcal E\in CS(X)$ set $\Psi(\mathcal E)=F^\mathcal E$. Thus, we have maps $\Phi:I(M(X))\to CS(X)$ and $\Psi:CS(X)\to I(M(X))$.

\begin{thm}
$\Phi\circ\Psi$ is the identity map on $CS(X)$; $\Psi\circ\Phi$ is not the identity map on $I(M(X))$.

\end{thm}

The proof consists of the following Lemmas \ref{L4} and \ref{L5}

\begin{lem}\label{L4}
$\Phi\circ\Psi$ is the identity map on $CS(X)$.

\end{lem}
\begin{proof}
$\Phi\circ\Psi(\mathcal E)$ is the coarse structure generated by all $\mathcal E_d$, $[d]\in F^\mathcal E$, therefore $\Phi\circ\Psi(\mathcal E)\subset\mathcal E$.
     
Let $E\in\mathcal E$, and let $\mathcal E'$ be the coarse structure generated by $E$, $E\circ E$, $E\circ E\circ E$, etc. As it is countably generated, there exists a metric $d$ on $X$ such that $\mathcal E'=\mathcal E_d$, and it is obvious that 1) $\mathcal E'\subset\mathcal E$, and 2) $E\in\mathcal E'$. Then, 1) implies $[d]\in F^\mathcal E$, and 2) implies $\mathcal E\subset\Phi\circ\Psi(\mathcal E)$.
\end{proof} 


Let $\delta_1\prec\delta_2\prec\cdots$ be a strictly increasing sequence in $M(X)$, and let $\mathcal E_{\delta_n}$ be the corresponding coarse structures. Let $\mathcal E=\cup_{n\in\mathbb N}\mathcal E_{\delta_n}$ be the coarse structure generated by all $\mathcal E_{\delta_n}$. Define $F_1$ as the subset of $M(X)$ consisting of all equivalence classes of metrics dominated by $\mathcal E$. As $\mathcal E$ is countably generated, it is metrizable, hence $F_1$ is a principal ideal.
Set $F_2=\cup_{n\in\mathbb N}F_{\delta_n}$. 

\begin{lem}\label{L5}
$F_1\neq F_2$, but $\Phi(F_1)=\Phi(F_2)$.

\end{lem}
\begin{proof}
Clearly, $F_2\subset F_1$. Suppose that $F_1\subset F_2$. As $F_1$ is principal, there is $\alpha\in M(X)$ such that $F_1=F_\alpha$. Then $\alpha\in F_2$, which means that there exists $n\in\mathbb N$ such that $\alpha\preceq\delta_n$. But then $\delta_{n+1}\notin F_\alpha$ --- a contradiction. In particular, $F_2$ is not principal.

By definition, we have $\Phi(F_1)=\mathcal E$. But $\Phi(F_2)=\cup_{n\in\mathbb N}\mathcal E_{\delta_n}=\mathcal E$ too, hence $\Phi(F_1)=\Phi(F_2)$. 
\end{proof}

This shows that there exist ideals that don't come from coarse structures.

The following Proposition allows to check whether $\Phi(F)=\Phi(G)$ for ideals $F,G\in I(M(X))$.

\begin{prop}
Let $F,G\in I(M(X))$. $\Phi(F)\subset\Phi(G)$ iff for any $\alpha\in F$, any $a\in\alpha$, and any $R>0$ there exist $b\in\mathcal M(X)$ with $\beta\in G$, and $S>0$ such that $a(x,y)<R$ implies $b(x,y)<S$ for any $x,y\in X$. 

\end{prop}
\begin{proof}
1. Assume that $\Phi(F)\subset\Phi(G)$. Let $a\in\mathcal M(X)$, $\alpha=[a]\in F$. Given $R>0$, let $E_R(a)=\{(x,y)\in X\times X:a(x,y)\leq R\}$. By assumption, $E_R(a)\in\Phi(G)$. But $\Phi(G)=\cup_{\beta\in G}\mathcal E_\beta$, hence there exists $\beta\in G$ such that $E_R(a)\in \mathcal E_\beta$. Let $b\in\beta$. Then there exists $S>0$ such that $E_R(a)\subset E_S(b)$.

2. To prove the other implication, let $\alpha=[a]\in F$, and let $E\in\mathcal E_\alpha$. Then there exists $R>0$ such that $E\subset E_R(a)$. By the assumption, there exists $S>0$, $\beta\in G$ and $b\in\beta$ such that $E_R(a)\subset E_S(b)$. Thus, $E\subset E_S(b)\in\mathcal E_\beta\subset\Phi(G)$, hence $\Phi(F)\subset\Phi(G)$. 
\end{proof}

\section{Uniform Roe algebras for ideals}

Recall that the uniform Roe algebras are monotone with respect to the partial order on $M(X)$: if $\alpha,\beta\in M(X)$, $\alpha\preceq\beta$, then $C^*_u(X,\alpha)\subset C^*_u(X,\beta)$. 

Any ideal is a directed partially ordered set, so, given $F\in I(M(X))$, we can define $C^*_u(X,F)=\lim_{\alpha\in F}C^*_u(X,\alpha)$. Clearly, $C^*_u(X,F_\alpha)=C^*_u(X,\alpha)$ for any $\alpha\in M(X)$.

\begin{prop}\label{Prop}
If $C^*_u(X,F)=C^*_u(X,\alpha)$ then $\alpha$ is the least upper bound for $F$.

\end{prop}
\begin{proof}
Suppose the contrary. If $\alpha$ is not the least upper bound for $F$ then there are two possibilities: either $\alpha$ is not an upper bound (and then there exists $\beta\in F$ such that $\beta\npreceq \alpha$) or $\alpha$ is not the least upper bound (and then there exists another upper bound $\gamma\in M(X)$ such that $\beta\preceq \gamma\prec \alpha$ for any $\beta\in F$).

In the first case $\beta\npreceq \alpha$ implies that $C^*_u(X,\beta)\nsubseteq C^*_u(X,\alpha)$, hence $C^*_u(X,F)\nsubseteq C^*_u(X,\alpha)$.

In the second case we have $C^*_u(X,F)\subseteq C^*_u(X,\gamma)\subset C^*(X,\alpha)$, and as $\gamma\neq \alpha$, $C^*_u(X,\gamma)$ is strictly smaller than $C^*_u(X,\alpha)$. 

In both cases we get a contradiction with $C^*_u(X,F)=C^*_u(X,\alpha)$. 
\end{proof}

For a coarse structure $\mathcal E$ one can define the uniform Roe algebra $C^*_u(X,\mathcal E)$ as for a metric space. Our standard references on uniform Roe algebras are \cite{Roe,Nowak-Yu}. Let $\{\delta_x\}_{x\in X}$, $\delta_x(y)=\left\lbrace\begin{array}{cl}1,&\mbox{if\ }y=x;\\0,&\mbox{otherwise,}\end{array}\right.$ be the standard basis for the Hilbert space $l_2(X)$. The uniform Roe algebra is generated by bounded operators $T\in\mathbb B(l_2(X))$ such that $T_{xy}=\langle\delta_x,T\delta_y\rangle=0$ for $(x,y)\notin E$ for some $E\in\mathcal E$. When $\mathcal E=\mathcal E_[a]$ is the coarse structure generated by a metric $a\in\mathcal M(X)$ then $C^*_u(X,\mathcal E_[a])$ is the uniform Roe algebra for the metric space $(X,a)$, which depends not on the metric itself, but on its equivalence class $\alpha=[a]$. We shall write $C^*_u(X,\alpha)$ for the latter algebra. 

\begin{prop}
For any ideal $F\in I(M(X))$, one has
$C^*_u(X,F)=C^*_u(X,\Phi(F))$.

\end{prop}
\begin{proof}
By definition, $C^*_u(X,F)$ is the direct limit $\injlim_{\alpha\in F}C^*_u(X,\alpha)$, which is isomorphic to the norm closure of $\cup_{\alpha\in F}C^*_u(X,\alpha)$ in $\mathbb{B}(l_2(X))$.

Since $F$ is an ideal, the family of coarse structures $\{\mathcal{E}_\alpha\}_{\alpha\in F}$ is directed under inclusion, and their union is itself a coarse structure, so $\Phi(F)=\cup_{\alpha\in F}\mathcal{E}_\alpha$.

An operator $T$ belongs to the algebraic uniform Roe algebra $\mathbb{C}_u[X,\Phi(F)]$ if and only if its support $\supp(T)$ is contained in some entourage $E\in\Phi(F)$. This means $\supp(T)\in\mathcal{E}_\alpha$ for some $\alpha\in F$, which is exactly the condition that $T$ belongs to the algebraic uniform Roe algebra $\mathbb{C}_u[X,\alpha]$. Thus, $\mathbb{C}_u[X,\Phi(F)]=\cup_{\alpha\in F}\mathbb{C}_u[X,\alpha]$. Hence the claim follows.
\end{proof}

If $\mathcal E$ is not metrizable then $C^*_u(X,\mathcal E)$ does not equal $C^*_u(X,d)$ for any metric $d\in\mathcal M(X)$.

The uniform Roe algebra $C^*_u(X,\mathcal E^{BG})$ for the ideal from the coarse structure of bounded geometry $\mathcal E^{BG}$ was described in \cite{Man-Zhu}.

\end{document}